\documentclass{amsart}
\usepackage{amsfonts,amssymb,amsmath,amsthm}
\usepackage{url}
\usepackage{enumerate}
\usepackage[all,cmtip]{xy}

\newtheorem{thrm}{Theorem}[section]

\newtheorem{prop}[thrm]{Proposition}
\newtheorem{cor}[thrm]{Corollary}
\theoremstyle{definition}
\newtheorem{definition}[thrm]{Definition}

\newtheorem{remark}[thrm]{Remark}

\numberwithin{equation}{section}

\newcommand{\fgrade}{\operatorname{fgrade}}

\newcommand{\Spec}{\operatorname{Spec}}

\newcommand{\Ext}{\operatorname{Ext}}
\newcommand{\Supp}{\operatorname{Supp}}

\newcommand{\Hom}{\operatorname{Hom}}
\newcommand{\Fdepth}{\operatorname{F-depth}}

\newcommand{\depth}{\operatorname{depth}}

\newcommand{\vpl}{\operatornamewithlimits{\varprojlim}}

\newcommand{\vil}{\operatornamewithlimits{\varinjlim}}

\newcommand{\fm}{\frak{m}}

\newcommand{\fn}{\frak{n}}

\author{ Majid Eghbali}

\thanks{This research was in part supported by a grant from IPM}

\address{School of Mathematics, Institute for Research in Fundamental Sciences (IPM), P. O. Box: 19395-5746,
Tehran-Iran.} \email{m.eghbali@yahoo.com}



\begin{document}

\title[ Cohomological invariants of local rings]{On cohomological invariants of local rings in positive characteristic
}

\begin{abstract}
The Frobenius depth denoted by $\Fdepth$ defined by
Hartshorne-Speiser in 1977 and later by Lyubeznik in 2006, in a
different way, for rings of positive characteristic. The aim of the
present paper is to compare the $\Fdepth$ with formal grade, and
$\depth$ to shed more light on the notion of Frobenius depth from a
different point of view.
\end{abstract}
 \maketitle

\section{Introduction} \label{sect1}

Let $Y$ be a closed subscheme of $\mathbb{P}^n_k$, the projective
space over a field $k$ of characteristic $p>0$. Vanishing of
$H^i(\mathbb{P}^n-Y,\mathcal{F})$ for all coherent sheaves
$\mathcal{F}$ was asked by Grothendieck (\cite{Gr68}). Among the
attempts to answer the mentioned question Hartshorne and Speicer in
\cite{HS} used the notion of Frobenius depth of $Y$ to give an
essentially complete solution to this problem.

To be more precise, Let $Y$ be a Noetherian scheme of finite
dimension, whose local rings are all of characteristic $p>0$. Let $y
\in Y$ be a (not necessarily closed) point. Let $d(y)$ be the
dimension of the closure $\{y\}^-$ of the point $y$. Let
$\mathcal{O}_y$ be the local ring of $y$, let $k_0$ be its residue
field, let $k$ be a perfect closure of $k_0$, and let
$\widehat{\mathcal{O}}_y$, be the completion of $\mathcal{O}_y$.
Choose a field of representatives for $k_0$ in
$\widehat{\mathcal{O}}_y$. Then we can consider
$\widehat{\mathcal{O}}_y$ as a $k_0$-algebra, and we let $A_y$ be
the local ring $\widehat{\mathcal{O}}_y \otimes_{k_0} k$ obtained by
base extension to $k$. Let $Y_y=\Spec A_y$ and let $P$ denote its
closed point. So, the Frobenius depth of $Y$ is denoted by $\Fdepth
Y$ is the largest integer $r$ (or $+\infty$) such that for all
points $y \in Y$, one has $H^i_P(Y_y,\mathcal{O}_{Y_y})_s=0$ (the
stable part of $H^i_P(Y_y,\mathcal{O}_{Y_y})$) for all $i < r-d(y)$.

 From the local
algebra point of view, Grothendieck's problem is stated to find
conditions under which  $H^i_I(M)=0$ for all $i>n\ (n \in
\mathbb{Z})$ and all $A$-modules $M$, where $A$ is a commutative
Noetherian local ring and $I \subset A$ is an ideal. For an
$A$-module $M$, we denote by $H^i_I(M)$ the $i$th local cohomology
module of $M$ with respect to $I$. For more details the reader may
consult \cite{Gr67} and \cite{Br-Sh}. From the celebrated result of
Hartshorne (cf. \cite[pp. 413]{Hart68}), it is enough to find
conditions for the vanishing of $H^i_I(A)$. In this direction, for a
local ring $(A,\fm)$, Lyubeznik in \cite{Lyu06} using the Frobenius
map from $ H^i_{\fm}(A)$ to itself defined the Frobenius depth of
$A$ denoted by $\Fdepth A$ as the smallest $i$ such that for every
iteration of Frobenius map, $H^i_{\fm}(A)$ does not go to zero. It
is note worthy to say that the Lyubeznik's $\Fdepth$ coincides with
the notion of $\Fdepth$ defined by Hartshorne and Speiser, whenever
$A$ admits a surjection from a regular local ring and $Y=\Spec A$
(cf. \cite[Corollary 6.3]{Lyu06}).

Consider the family of local cohomology modules $\{H^i_{\fm}
(M/I^nM)\}_{n \in \mathbb{N}}$ where, $(A,\fm)$ is not necessarily
of characteristic $p>0$. For every $n \in \mathbb{N}$ there is a
natural homomorphism
$$H^i_{\fm} (M/I^{n+1}M) \rightarrow H^i_{\fm} (M/I^nM)$$
(induced from the natural projections $M/I^{n+1}M \rightarrow
M/I^nM$) such that the family forms a projective system. The
projective limit ${\vpl}_nH^i_{\fm}(M/I^n M)$ is called the $i$th
formal local cohomology of $M$ with respect to $I$ (cf. \cite{Sch}).
Formal local cohomology modules were used by Peskine and Szpiro in
\cite{P-S} when $A$ is a regular ring of prime characteristic. It is
noteworthy to mention that if $U = \Spec(A) \setminus \{\fm\}$ and
$( \widehat{U},\mathcal{O}_{\widehat{u}})$ denote the formal
completion of $U$ along $V(I) \setminus \{\fm \}$ and also
$\widehat{\mathcal{F}}$ denotes the
$\mathcal{O}_{\widehat{u}}$-sheaf associated to ${\vpl}_n M/I^n M$,
they have described the formal cohomology modules
$H^i(\widehat{U},\mathcal{O}_{\widehat{u}})$ via the isomorphisms
$H^i(\widehat{U},\mathcal{O}_{\widehat{u}})\cong
{\vpl}_nH^i_{\fm}(M/I^n M)$, $i \geq 1$. See also \cite[proposition
(2.2)]{Og} when $A$ is a Gorenstein ring.

The formal grade, $\fgrade(I,M)$, is defined as the index of the
minimal nonvanishing formal cohomology module, i.e., $\fgrade(I,M) =
\inf \{i \in \mathbb{Z}|\ \ {\vpl}_nH^i_{\fm}(M/I^n M) \neq 0\}$.
One way to check out vanishing of local cohomology modules is the
following duality
\begin{equation}\label{duality}
    {\vpl}_nH^i_{\fm}(A/I^n) \cong \Hom_A(H^{\dim A-i}_{I}(A),E(A/\fm)),
\end{equation}
where $(A,\fm)$ is a Gorenstein local ring and $E(A/\fm)$ denotes
the injective hull of the residue field (cf. \cite[Remark
3.6]{Sch}). To be more precise, in this case the last non vanishing
amount of $H^{i}_{I}(A)$ may be described with the $\fgrade(I,A)$.
Thus, it motivates us to consider the invariants $\Fdepth$ and
$\fgrade$ to shed more light on the notion of Frobenius depth from a
different point of view. For this reason, in Section $2$, we bring
some auxiliary results and among them we examine the structure of
${\vpl}_nH^i_{\fm}(A/I^n)$ as a unit $A[F^e]$-module (Theorem
\ref{unit}). Moreover, it has a structure of $D$-modules. In Section
$3$, we concentrate on $\Fdepth$ and reprove some known results and
then compare it with the formal grade and depth, (cf. Theorem
\ref{comparison} and Corollary \ref{connected}).

\section{Auxiliary Results}

Throughout this section all rings are assumed to contain a field of
positive characteristic. The symbol $A$ will always denote a
commutative Noetherian ring of finite characteristic. We adapt the
notation from \cite{Bl} and except for notation we mostly follow
Lyubeznik \cite{Lyu97}. We let $F=F_A$ the Frobenius map on $A$,
that is $F:A \rightarrow A$, with $a \mapsto a^p,\ a \in A$. We
denote the $e$th iterate of the Frobenius map by $A^e$ which is the
$A-A$-bimodule. As a left $A$-module it is $A$ and as a right
$A$-module we have $m.a=a^{p^e}m$ for $m \in A^e$. We say $A$ is
F-finite, whenever $A^e$ is a finitely generated right $A$-module.

\begin{remark}\label{Hart-Speis}
Let us recall from \cite[Proposition 1.1(a)]{HS} that for a ring $A$
which is either a localization of an algebra of finite type over a
perfect field $k$, or a complete local ring containing a perfect
field $k$ as its residue field, then $A$ is $F$-finite.
\end{remark}

In the present section, among our results we recall various results
due to Hartshorne-Speiser \cite{HS}, Peskine-Szpiro \cite{P-S},
Lyubeznik \cite{Lyu97} and Blickle \cite{Bl}.

Peskine and Szpiro in \cite{P-S} defined the Frobenius functor as
follows:

\begin{definition} The Frobenius functor is the right exact functor
from $A$-modules to $A$-modules given by
$$F^{\ast}_A M:=A^1 \otimes_A M.$$
Its $e$th power is $F^{e\ast}_A M=A^e \otimes_A M$. For brevity we
often write $F^{e\ast}$ for $F^{e\ast}_A $ when there is no
ambiguity about the ring $A$.
\end{definition}

It follows from the definition that $F^{e\ast}$ commutes with direct
sum, direct limit and localization. By a theorem of Kunz \cite{Ku69}
the Frobenius functor is flat whenever $A$ is a regular ring, hence
in this case $F^{e\ast}$ will be exact and immediately one has
$F^{e\ast} A=A^e \otimes_A A \cong A$ and $F^{e\ast} I=I^{[p^e]}$ an
ideal of $A$ generated by $p^e$th powers of the elements of $I$.

\begin{definition} \label{def.unit} An $A[F^e]$-module is an $A$-module $M$ together with an $A$-linear map
$$\mathcal{V}^e_M:F^{e\ast}_AM \rightarrow M.$$
A morphism between two $A[F^e]$-modules $(M,\mathcal{V}^e_M)$ and
$(N,\mathcal{V}^e_N)$ is an $A$-linear map $\varphi : M \rightarrow
N$ such that the following diagram commutes:

$$\xymatrix{
F^{e\ast}M \ar[d]^{\mathcal{V}^e_M} \ar[r]^{F^{e\ast}(\varphi)} &F^{e\ast}N\ar[d]^{\mathcal{V}^e_N}\\
M \ar[r]^{\varphi} &N}.$$ In fact, we can consider $F^eM$ as a
$p^e$-linear map from $M \rightarrow M$; as such it is not
$A$-linear but we have $F^e(am) = a^{[p^e]}F^e(m),\ a \in A^e, m \in
M$. Furthermore, an $A[F^e]$-module $(M,\mathcal{V}^e_M)$ is called
a unit $A[F^e]$-module if $\mathcal{V}^e_M$ is an isomorphism (cf.
\cite[page 16]{Bl}).
\end{definition}

\begin{remark}
 As we have seen above, $A$ is a unit $A[F^e]$-module, whenever $A$
is regular but $I$ is not in general. For a multiplicatively closed
subset $S$ of $A$, the structural map $\mathcal{V}^e_{S^{-1}A}:A^e
\otimes S^{-1}A \rightarrow S^{-1}A$ is an isomorphism (see
\cite{Bl}).
\end{remark}

The following definition introduced in \cite{HS}:

\begin{definition} \label{forgetful} Let $(M,\mathcal{V}^e)$ be an $A[F^e]$-module.
We define $G(M)$ as the inverse limit generated by the structural
map $\mathcal{V}^e$, i.e.
$$G(M):={\vpl}(\cdots \rightarrow F^{3e\ast}M \stackrel{F^{2e\ast}\mathcal{V}^e}{\longrightarrow}
 F^{2e\ast}M \stackrel{F^{e\ast}\mathcal{V}^e}{\longrightarrow} F^{e\ast}M \stackrel{\mathcal{V}^e}{\longrightarrow} M).$$
\end{definition}
Note that there are natural maps $\pi_e: G(M) \rightarrow
F^{e\ast}M$. Moreover, the maps $F^{e\ast}\pi_r: F^{e\ast}G(M)
\rightarrow F^{e(r+1)\ast}M$ are compatible with the maps defining
$G(M)$ and thus by the universal property of inverse limits, the map
$F^{e\ast}G(M) \rightarrow G(M)$ defines the natural $A[F^e]$-module
structure on $G(M)$.

\begin{prop}  \label{4.1}(\cite[Proposition 1.2]{HS} and \cite[Proposition 4.1]{Bl}) Let $A$ be regular and $F$-finite and $M$ an
$A[F^e]$-module. Then $G(M)$ is a unit $A[F^e]$-module.
\end{prop}

In order to extend the Matlis duality functor $D(-)=\Hom(-, E_A)$
where, $E_A$ is the injective hull of the residue field, the functor
$\mathcal{D}$ from $A[F^e]$-modules to $A[F^e]$-modules is defined
as follows ( \cite[Section 4]{Bl}):

Let $(M, \mathcal{V}^e)$ be an $A[F^e]$-module. We define
$$\mathcal{D}(M)={\vil}(D(M) \stackrel{D(\mathcal{V}^e)}{\longrightarrow}
 D(F^{e\ast}M) \stackrel{D(F^{e\ast}(\mathcal{V})}{\longrightarrow} D(F^{2e\ast}M) \longrightarrow \ldots).$$
An element $m \in M$ is called $F$-nilpotent if $F^{re}(m)=0$ for
some $r$. Then $M$ is called $F$-nilpotent if $F^{er}(M)=0$ for some
$r \geq 0$. It is possible that every element of $M$ is
$F$-nilpotent but $M$ itself is not.

Below, we recall some properties of the functor $\mathcal{D}$.

\begin{prop} \label{4.16-4.17}Let $A$ be a complete regular local ring.

(a) On the subcategory of $A[F^e]$-modules which are cofinite (i.e.
satisfy the descending chain condition for submodules) as
$A$-modules $\mathcal{D}$ is exact and its values are finitely
generated unit $A[F^e]$-modules (cf. \cite[Theorem 4.2(i)]{Lyu97}
and \cite[Proposition 4.16]{Bl}).

(b) Let $M$ be an $A[F^e]$-module that is finitely generated or
cofinite as an $A$-module. Then $\mathcal{D}(\mathcal{D}(M))\cong
G(M)$ (cf. \cite[Proposition 4.17]{Bl}).

(c) Let $M$ be an $A[F^e]$-module which is a cofinite $A$-module.
Then $M$ is $F$-nilpotent if and only if $\mathcal{D}(M) = 0$ (cf.
\cite[Theorem 4.2(ii)]{Lyu97} and \cite[Proposition 4.20]{Bl}).
\end{prop}

Let $A$ be a regular local ring and let $I$ be an ideal of $A$. As
we have seen (following \cite[page 19]{Bl}) that $A$ is an
$A[F^e]$-module and come down this structure to its localizations.
The local cohomology modules $H^i_I(A)$ of $A$ with support in $I$
can be calculated as the cohomology modules of the \v{C}ech complex
$$\check{C}(A; x_1,\ldots, x_n) = A \rightarrow A_{x_i} \rightarrow A_{x_ix_j} \rightarrow A_{x_1x_2\cdots x_n}$$
where, $I$ is generated by $x_1,x_2,\ldots, x_n$. Thus, the modules
$H^i_I(A)$ are $A[F^e]$-modules as the category of $A[F^e]$-modules
is Abelian. Moreover, the $H^i_I(A)$ are unite $A[F^e]$-modules for
all $i \in \mathbb{Z}$, but the modules $H^i_{\fm}(A)$ are not unit
in general. For formal local cohomology modules, the situation is a
bit more complicated, however, we show that these kind of modules
have unite $A[F^e]$-modules structure, where $A$ is $F$-finite.

\begin{thrm} \label{unit} Let $(A,\fm)$ be a
regular $F$-finite local ring. Then
$$G(H^i_{\fm}(A/I)) \cong
{\vpl}_nH^i_{\fm}(A/I^{n}),\ i \in \mathbb{Z}$$ which is a unit
$A[F^e]$-module. In particular, $$ {\vpl}_nH^i_{\fm}(A/I^{n}) \cong
H^i_{\fm}(\hat{A^I}),\ i \in \mathbb{Z}$$ as $A[F^e]$-modules, where
$\hat{A^I}$ is the completion of $A$ along $I$.
%the endomorphism ring of
%${\vpl}_nH^i_{\fm}(A/I^{n})$, i.e.
%$\End_R({\vpl}_nH^i_{\fm}(A/I^{n}))$ is a unit $A[F^e]$-module for
%all $i$.
\end{thrm}

\begin{proof}
%Since $R$ is a regular local ring then $R^e$ is a
%finitely generated free right $R$-module. Thus $R$ is
%$F$-finite???????.
By what we have seen above,  $H^i_{\fm}(A/I)$ is an $A[F^e]$-module
for all $i \in \mathbb{N}$. Now, we may apply functor $G(-)$ to
$H^i_{\fm}(A/I)$:
$$G(H^i_{\fm}(A/I))=
{\vpl}(\cdots \stackrel{F^{2e\ast}\mathcal{V}^e}{\longrightarrow}
 H^i_{\fm}(A/I^{[p^{2e}]}) \stackrel{F^{e\ast}\mathcal{V}^e}{\longrightarrow} H^i_{\fm}(A/I^{[p^{e}]}) \stackrel{\mathcal{V}^e}
 {\longrightarrow} H^i_{\fm}(A/I)).$$
By virtue of Proposition \ref{4.1}, $G(H^i_{\fm}(A/I))$ is a unite
$A[F^e]$-module. Notice that, the right hand side is nothing but
${\vpl}_eH^i_{\fm}(A/I^{[p^e]})$.

On the other hand, one has ${\vpl}_nH^i_{\fm}(A/I^{n}) \cong
\Hom_A(H^{\dim A-i}_{I}(A),E)$, where $E:=E(A/\fm)$ is the injective
hull of the residue field. The natural map
$$F^{e \ast} \Hom_A(H^{\dim A-i}_{I}(A),E) \rightarrow \Hom_A(H^{\dim A-i}_{I}(A),E)$$
by sending $r \otimes \varphi$ to $rF^{e \ast}(\varphi)$ is an
isomorphism of $A[F^e]$-modules ($r \in A$ and $\varphi \in
\Hom_A(H^{\dim A-i}_{I}(A),E)$). To this end note that $H^{\dim
A-i}_{I}(A)$ and $E \cong H^{\dim A}_{\fm}(A)$ carry natural unite
$A[F^e]$-structure. Thus, ${\vpl}_nH^i_{\fm}(A/I^{n})$ is a unit
$A[F^e]$-module for each $i \in \mathbb{Z}$.

In order to complete the proof, it is enough to show that
${\vpl}_eH^i_{\fm}(A/I^{[p^e]}) \cong {\vpl}_nH^i_{\fm}(A/I^{n})$.
For this reason, consider the decreasing family of ideals
$\{I^{[p^e]}\}_e$. Clearly, its topology is equivalent to the
$I$-adic topology on $A$. Thus, by \cite[Lemma 3.8]{Sch} there
exists a natural isomorphism
$${\vpl}_eH^i_{\fm}(A/I^{[p^e]}) \cong
{\vpl}_nH^i_{\fm}(A/I^{n})$$ for all $i \in \mathbb{Z}$.

The last part follows by \cite[Proposition 2.1]{HS}.
\end{proof}

\begin{remark} It should be noted that with the assumptions of Theorem \ref{unit},
for all $i \in \mathbb{Z}$, the module ${\vpl}_nH^i_{\fm}(A/I^{n})$
carries a natural $D_A$-module and $\mathcal{V}^e$ (cf. Definition
\ref{def.unit}) is a map of $D_A$-modules. The interested reader may
consult \cite[Section 5]{Lyu97} and \cite[Chapter 3]{Bl}.
\end{remark}

%For an $A[F^e]$-module $M$, the image $F^e(M)$ is not a submodule of
%$M$, in general, but in the case $A$ containing a perfect field $k$
%it is a $k$-vector space \cite[page 47]{HS}. We have $F^e(M)
%\supseteq F^{2e}(M) \supseteq \ldots$ and we call the intersection
%$\bigcap_{r \geq1}F^{re}(M)$ the stable part of $M$ and is denoted
%by $M_s$ which is a $k$-vector space. With the same idea as in the
%proof of \cite[Corollary 2.2]{HS} we are able to consider the stable
%part of ${\vpl}_nH^i_{\fm}(A/I^{n})$.

%\begin{prop} ?????????? Let $(A,m)$ be a complete regular local ring containing its residue field $k$ which is perfect.
%Let $I$ be an ideal of $A$. Further assume that
%${\vpl}_nH^i_{\fm}(A/I^{n})$ is a cofinite $A$-module, $i \in
%\mathbb{Z}$. Then,
% \item [(a)] $H^i_{\fm}(A)$ is a unit
% the stable part of ${\vpl}_nH^i_{\fm}(A/I^{n})$ is isomorphic to
%$H^i_{\fm}(A/I)_s$ which is a finite dimensional $k$-vector space,
%on which $F^e$ acts bijectively.
%\end{prop}

%\begin{proof} First note that $A$ is $F$-finite. Both
%$H^i_{\fm}(A/I)$ and ${\vpl}_nH^i_{\fm}(A/I^{n})$ are  cofinite
%$A$-modules ($i \in \mathbb{Z}$), hence by \cite[1.12]{HS} their
%stable parts $H^i_{\fm}(A/I)_s$ respectively
%$({\vpl}_nH^i_{\fm}(A/I^{n}))_s$ are finite dimensional $k$-vector
%spaces on which $F^e$ acts bijectively on them. Thus, they are unit
%$k[F^e]$-modules. By virtue of \cite[1.2(a) and 1.2(b)]{HS} applied
%to $k$ we have
%$$({\vpl}_nH^i_{\fm}(A/I^{n}))_s \cong G(H^i_{\fm}(A/I))_s \cong {\vpl}_nH^i_{\fm}(A/I^{n})_s$$
%and
%$$G(H^i_{\fm}(A/I))_s \cong H^i_{\fm}(A/I)_s.$$
%Now, we are done.
%\end{proof}

\section{Frobenius depth} \label{sect2}

Let $A$ be a regular local $F$-finite ring of characteristic $p>0$
and let $I$ be an ideal of $A$. As we have seen in the previous
section the formal local cohomology modules, are unite
$A[F^e]$-modules. In the present section we use the unit $A[F^e]$
structure of ${\vpl}_nH^i_{\fm}(A/I^{n})$ in order to prove our
results.

\begin{prop} \label{f-nil} Let $(A,\fm)$ be a
regular local and $F$-finite ring. Then
${\vpl}_nH^i_{\fm}(A/I^{n})=0$ if and only if $H^i_{\fm}(A/I)$ is
$F$-nilpotent.
\end{prop}

\begin{proof}
% (see proof of \cite[Proposition 4.16]{Bl})
By the assumptions $A$ is $F$-finite that is $A^e$ is a finitely
generated $A$-module. Then tensoring with $A^e$ commutes with the
inverse limit, as $A^e$ is a free right $A$-module (cf.
\cite[Proposition 1.1(b)]{HS}). Thus, we have
$$A^e \otimes_A \widehat{A}=A^e \otimes_A ({\vpl}_n A/I^n) \cong {\vpl}_n A^e/A^eI^n= {\vpl}_n A^e/I^{n[p^e]}A^e \cong \widehat{A^e}.$$
On the other hand, since the Frobenius action is the same in both
$H^i_{\fm}(A/I)$ and $H^i_{\widehat{\fm}}(\widehat{A}/I
\widehat{A})$, so we may assume that $A$ is a complete regular local
$F$-finite ring.

As $H^i_{\fm}(A/I)$ is an $A[F^e]$-module which is a cofinite
$A$-module, then $\mathcal{D}(H^i_{\fm}(A/I))$ is a finitely
generated unit $A[F^e]$-module (cf. \ref{4.16-4.17}(a)) and
therefore
$$\mathcal{D}(\mathcal{D}(H^i_{\fm}(A/I)))\cong
D(\mathcal{D}(H^i_{\fm}(A/I))).$$ As the functor $D(-)$ transforms
direct limits to inverse limits, then
\begin{eqnarray*}
  D(\mathcal{D}(H^i_{\fm}(A/I)))\!\!\!\! &=&\!\!\!\! D({\vil}(D(H^i_{\fm}(A/I))
  \rightarrow D(F^{e\ast} H^i_{\fm}(A/I)) \rightarrow D(F^{2e\ast} H^i_{\fm}(A/I)) \rightarrow \cdots)  \\
            \!\!\!\! & \cong &\!\!\!\! {\vpl}(\cdots \rightarrow D(D(F^{2e\ast} H^i_{\fm}(A/I)))
  \rightarrow D(D(F^{e\ast} H^i_{\fm}(A/I))) \rightarrow
  D(D(H^i_{\fm}(A/I)))) \\
  \!\!\!\! & \cong &\!\!\!\! {\vpl}_eH^i_{\fm}(A/I^{[p^e]}).
\end{eqnarray*}
As we have seen in the proof of Theorem \ref{unit},
${\vpl}_eH^i_{\fm}(A/I^{[p^e]}) \cong {\vpl}_nH^i_{\fm}(A/I^{n})$.
Therefore,
 $H^i_{\fm}(A/I)$ is $F$-nilpotent if and only if
$\mathcal{D}(H^i_{\fm}(A/I))=0$ (\ref{4.16-4.17}(c)) if and only if
${\vpl}_nH^i_{\fm}(A/I^{n})=0$.
\end{proof}

\begin{remark}
Notice that in the Proposition \ref{f-nil} the $F$-finiteness of $A$
is vital, because it guarantees the $A[F^e]$ structure of the
modules ${\vpl}_nH^i_{\fm}(A/I^{n})$. However, in the light of
\cite[Lemma 4.12]{Bl} if $H^{\dim A-i}_I(A)$ is cofinite, then the
module ${\vpl}_nH^i_{\fm}(A/I^{n})$ is $A[F^e]$-module.
\end{remark}

\begin{cor} \label{Lyu3.2}(\cite[Corollary 3.2]{Lyu06}) Let $(A,\fm)$ be a
 regular local ring and $I$ an ideal of $A$. Then $H^{\dim A-i}_{I}(A)=0$ if
and only if $F^{er}:H^i_{\fm}(A/I) \rightarrow H^i_{\fm}(A/I)$ is
the zero map for some $r>0$.
\end{cor}

\begin{proof} As $\widehat{A}$ is a faithful flat $A$-module then by passing to the completion
we may assume that $A$ is complete regular local ring. Let $H^{\dim
A-i}_{I}(A)=0$, so it is cofinite. Then by the duality
(\ref{duality}) in the introduction, one has
${\vpl}_nH^i_{\fm}(A/I^{n})=0$. Hence, Proposition \ref{f-nil}
implies that $H^i_{\fm}(A/I)$ is $F$-nilpotent.

Conversely, assume that $H^i_{\fm}(A/I)$ is $F$-nilpotent. Then by
Propostion \ref{4.16-4.17}(c) $\mathcal{D}(H^i_{\fm}(A/I))=0$ and
therefore one has $H^{\dim A-i}_{I}(A)=0$. To this end note that,
$\mathcal{D}(H^i_{\fm}(A/I^{[p^e]})) \cong
\Ext^{i}_A(A/I^{[p^e]},A)$ for all $e \geq 0$ and the Frobenius
powers of $I$ are cofinal with its ordinary powers.
\end{proof}

In the light of Corollary \ref{Lyu3.2}, Lyubeznik \cite{Lyu06}
defined the $\Fdepth$ of a local ring in order to give a solution to
Grothendieck's Problem.

\begin{definition} Let $(A,\fm)$ be a local ring. The $\Fdepth$ of
$A$ is the smallest $i$ such that $F^{er}$ does not send
$H^i_{\fm}(A)$ to zero for any $r$.
\end{definition}

One of elementary properties of $\Fdepth$ shows that $\Fdepth A$ is
equal to the $\Fdepth$ of its $\fm$-adic completion, $\widehat{A}$
(cf. \cite[Proposition 4.4]{Lyu06}) because $H^i_{\fm}(A)\cong
H^i_{\widehat{\fm}}(\widehat{A})$. In the next result we give an
alternative proof of \cite[Lemma 4.2]{Lyu06} to emphasize that
$\Fdepth$ of $A$ is bounded above by its Krull dimension.

\begin{prop} \label{lyu4.2} Let $(A,\fm)$ be a
 local ring and $I$ an ideal of $A$. Then $F^{er}$ does not send $H^{\dim A}_{\fm}(A)$ to zero for any $r$.
 In particular, $0 \leq Fdepth A \leq \dim A$.
\end{prop}

\begin{proof} Since the Frobenius action is the same in both
$H^i_{\fm}(A)$ and $H^i_{\widehat{\fm}}(\widehat{A})$, so we may
assume that $A$ is a complete local ring. Thus, by the Cohen's
Structure Theorem $A \cong R/J$, where $R$ is a complete regular
local ring and $J \subset R$ and ideal. In the contrary, assume that
$H^{\dim A}_{\fm}(R/J)$ is $F$-nilpotent. Then, $\mathcal{D}(H^{\dim
A}_{\fm}(R/J))=0$ (cf. \ref{4.16-4.17}(c)) and with a similar
argument given in the proof of Corollary \ref{Lyu3.2}, one can
deduce the vanishing of $H^{0}_{J}(R)$. Hence, by virtue of
(\ref{duality}), in the introduction, one has ${\vpl}_nH^{\dim A}
_{\fm}(R/J^{n})=0$ which is contradiction (cf. \cite[Theorem
4.5]{Sch}).
\end{proof}

To investigate the other properties of $\Fdepth$, in the next
Theorem we compare the Frobenius depth of $A$ and $A^{sh}$. For this
reason, let us recall some preliminaries. For a local ring $A$ we
denote by $A^{sh}$ the strict Henselization of $A$. A local ring
$(A,\fm,k)$ is said to be strictly Henselian if and only if every
monic polynomial $f(T)\in A[T]$ for which $f(T) \in k[T]$ is
separable splits into linear factors in $A[T]$. For more advanced
expositions on this topic we refer the interested reader to \cite
{Mi}.

\begin{prop} \label{Hensel}  Let $(A,\fm)$ be a  complete local ring
 and $I$ be an ideal of $A$. Then
  $\Fdepth A=\Fdepth A^{sh}$.
\end{prop}

\begin{proof} First assume that $A$ is a regular local ring. We show that $\Fdepth A/I=\Fdepth
(A/I)^{sh}$. Put $\Fdepth A/I=t$. Then, by virtue of Corollary
\ref{Lyu3.2} one has $H^{i}_{I}(A)=0$ for all $i
> \dim A-t$. Due to the faithfully flatness of the inclusion $A
\rightarrow A^{sh}$ and the fact that $A^{sh}$ is a regular local
ring, it implies that $H^{i}_{I}(A^{sh})=0$ for all $i
> \dim A-t$. Again, using Corollary
\ref{Lyu3.2}, it follows that $\Fdepth (A/I)^{sh} \geq t$. To this
end note that $\dim A=\dim A^{sh}$ and $(A/I)^{sh}=A^{sh}/IA^{sh}$.
With the similar argument one has $\Fdepth A/I \geq \Fdepth
(A/I)^{sh}$. This completes the assertion.

Since $A$ is a complete local ring, then by virtue of Cohen's
Structure Theorem, $A$ is a homomorphic image of a regular local
ring $R$, i.e. $A=R/J$ for some ideal $J$ of $R$. Now, we are done
by the previous paragraph. To this end note that
$$\Fdepth A =\Fdepth R/J=\Fdepth (R/J)^{sh}=\Fdepth A^{sh}.$$
\end{proof}

\begin{remark} From Proposition \ref{Hensel} and \cite[Proposition 4.4]{Lyu06} one may
deduce that
$$\Fdepth A=\Fdepth \widehat{((\hat{A})^{sh})},$$
where, $A$ is a local ring.
\end{remark}
In the following, we compare the invariants $\depth$, $\Fdepth$ and
$\fgrade$.

\begin{thrm} \label{comparison}  Let $(A,\fm)$ be a  local $F$-finite ring which is a
homomorphic image of a regular local $F$-finite ring and let $I$ be
an ideal of $A$. Then
 $$ \fgrade(I,A) \leq \depth A \leq \Fdepth A.$$
\end{thrm}

\begin{proof} We have ${\vpl}_nH^i_{\fm \hat{A}}(\hat{A}/I^n \hat{A}) \cong {\vpl}_nH^i_{\fm}(A/I^n
A)$ (cf. \cite[Proposition 3.3]{Sch}) and $\Fdepth A \cong \Fdepth
\hat{A}$ (cf. \cite[Proposition 4.4]{Lyu06}). Thus, we may assume
that $A$ is a complete local ring. Suppose that $(R,\fn)$ is a
regular local $F$-finite ring with $A \cong R/J$ where, $J$ is an
ideal of $R$.

 Put $\fgrade(J,R)=t$. Then by
definition ${\vpl}_nH^i_{\fn}(R/J^n)=0$ for all $i <t$. It follows
from the Proposition \ref{f-nil} that $H^{i}_{\fn}(R/J)$ is
$F$-nilpotent for all $i<t$, i.e. $\Fdepth R/J \geq t$. With a
similar argument and again using Proposition \ref{f-nil} we have
$\fgrade(J,R)\geq \Fdepth R/J$. Thus, $\fgrade(J,R)= \Fdepth R/J$.

Now, we are done by \cite[Lemma 4.8(b)]{Sch}, \cite[Remark 3.1]{E13}
and the previous paragraph. To this end note that
\begin{eqnarray*}
  \fgrade(I,A) \leq \depth A  &\leq&  \fgrade(J,R)  \\
             &=& \Fdepth R/J \\
             &=& \Fdepth A.
\end{eqnarray*}
%$$\fgrade(I,A) \leq  \depth A \leq \fgrade(J,R) = \Fdepth R/J=\Fdepth A.$$
\end{proof}

Let $A$ be a complete local ring containing a perfect field $k$ as
its residue field, then $A$ satisfies the condition of Theorem
\ref{comparison}. To this end, note that by Remark \ref{Hart-Speis},
$A$ is F-finite. Furthermore, by virtue of Cohen's Structure Theorem
$A \cong R/J$ for some ideal $J \subset R$, where
$R=k[[x_1,\ldots,x_n]]$ is a regular $F$-finite ring.

\begin{cor} \label{connected}
 Let $(A,\fm)$ be a regular local and $F$-finite ring. Then we have
$$\depth A/I \leq \fgrade(I,A)=\Fdepth A/I \leq \dim A/I.$$
  %In particular, $\fgrade(I,A)=\fgrade(IA^{sh},A^{sh})$.
\end{cor}

\begin{proof}
The assertion follows from what we have seen in the proof of Theorem
\ref{comparison} and
 \cite[Remark 3.1]{E}.
%The last part of the claim follows by Proposition \ref{Hensel} and
%the first part. check $F$-finiteness and $\depth A/I=\depth
%A^{sh}/IA^{sh}$.
\end{proof}

%The assumption $\fgrade(I,A) \geq 2$ in Proposition \ref{connected}
%is not too big, as the vanishing of ${\vpl}_nH^0_{\fm}(A/I^{n})$ is
%well understood in \cite[Lemma 4.1]{Sch}. Moreover, the structure of
%${\vpl}_nH^1_{\fm}(A/I^{n})$ has been studied in \cite[Theorem
%1.4]{E13} for one-dimensional ideals.

\begin{remark}
(a) The necessary and sufficient conditions for small values of the
$\Fdepth$ of $A$ is given in \cite[Corollary 4.6]{Lyu06}.
\begin{itemize}
  \item [(1)] $\Fdepth A>0$ if and only if $\dim A>0$.
  \item [(2)] $\Fdepth A>1$ if and only if $\dim A \geq 2$ and the punctured spectrum of $A$ is formally
geometrically connected.
\end{itemize}
%It is interesting to note that $\fgrade(I,A)>0$ if and only if
%$\depth A>0$ without any restriction on the characteristic of the
%ring $A$.
Now, let $A$ be $F$-finite and $\depth A=0 < \dim A$.
 Then, one has $\fgrade(I,A)=0 < \Fdepth
 A$. It shows that the inequality in Theorem \ref{comparison}
can be strict.

(b) Keep the assumptions in Corollary \ref{connected}, if $\Fdepth
A/I>1$, then by \cite[Lemma 5.4]{Sch}, one has
$\Supp_{\hat{A}}(\hat{A}/I \hat{A})\setminus \{\hat{\fm}\}$ is
connected. To this end, note that $\hat{A}$ is a local ring so it is
indecomposable.
\end{remark}

\proof[Acknowledgements] This project is based on a question asked
by Professor Josep \`{A}lvarez-Montaner on the comparison of
$\Fdepth$ and $\fgrade$ when the author was visiting the Universitat
Polit\`{e}cnica de Catalunya, in September 2013. Hereby, I would
like to express my thanks to the Universitat Polit\`{e}cnica and
specially, to Josep, for their support and warm hospitality. I am
also strongly indebted with Josep for several fruitful discussions.

\end{document}